\documentclass[12pt,a4paper]{amsart}

\usepackage[pdftex,pagebackref]{hyperref}
\hypersetup{letterpaper=true,colorlinks=true}
\hypersetup{pdfpagemode=UseNone,urlcolor=blue}
\hypersetup{linkcolor=blue,citecolor=blue,pdfstartview=FitH}

\usepackage{amsmath,amsfonts,amssymb,amsthm}
\usepackage{graphicx}
\usepackage{enumerate}
\usepackage{verbatim}

\usepackage{color}

\usepackage{ulem}

\def\emph#1{\textit{#1}}

\input xy
\xyoption{all}

\newenvironment{acknowledgements}{\textbf{Acknowledgements.}}{}

\newenvironment{notations}{\textbf{Notations.}}{}

\renewcommand{\geq}{\geqslant}

\renewcommand{\tilde}{\widetilde}

\newcommand{\ra}{\rightarrow}
\newcommand{\lra}{\longrightarrow}

\newcommand{\hdot}{{\:\raisebox{3pt}{\text{\circle*{1.5}}}}}
\def\C{{\mathbb C}}

\renewcommand{\L}{\mathbb{L}}
\newcommand{\cO}{\mathcal{O}}
\newcommand{\cK}{\mathcal{K}}
\newcommand{\cE}{\mathcal{E}}
\newcommand{\cF}{\mathcal{F}}

\newcommand{\cZ}{\mathcal{Z}}

\newcommand{\cX}{\mathcal{X}}
\newcommand{\cY}{\mathcal{Y}}

\newcommand{\cM}{\mathcal{M}}
\newtheorem{theorem}{Theorem}[section]
\newtheorem{lemma}[theorem]{Lemma}
\newtheorem{proposition}[theorem]{Proposition}

\theoremstyle{definition}
\newtheorem{definition}[theorem]{Definition}
\newtheorem{example}[theorem]{Example}

\theoremstyle{remark}
\newtheorem{remark}[theorem]{Remark}

\renewcommand{\hom}{\textrm{Hom}}
\newcommand{\cok}{\textrm{coker}}

\newcommand{\spec}{\textrm{Spec}}

\newcommand{\ext}{\textrm{Ext}}

\newcommand{\im}{\textrm{Im}}

\newcommand{\chom}{\mathcal{H}\mathit{om}}
\newcommand{\cend}{\mathcal{E}\mathit{nd}}
\newcommand{\rhom}{\textrm{RHom}}
\newcommand{\crhom}{\mathcal{R}\mathcal{H}\mathit{om}}
\newcommand{\der}{\textrm{Der}}
\newcommand{\exal}{\textrm{Exal}}

\newcommand{\rpils}{\textrm{R}\pi_*}

\title[Moduli Spaces and Symplectic Stacks]{Moduli Spaces of Sheaves on K3 Surfaces and Symplectic Stacks}
\date{\today}
\author{Ziyu Zhang}
\address{Max Planck Institute for Mathematics, Vivatsgasse 7, 53111 Bonn, Germany}
\email{zhangzy@mpim-bonn.mpg.de}

\begin{document}

\maketitle

\begin{abstract}
We view the moduli space of semistable sheaves on a K3 surface as a global quotient stack, and compute its
cotangent complex in terms of the universal sheaf on the Quot scheme. Relevant facts on the classical and 
reduced Atiyah classes are reviewed. We also define the notion of a symplectic
stack, and show that it includes all moduli stacks of semistable sheaves on K3 surfaces. 
\end{abstract}

\tableofcontents

%======================================
%       Section One: Introduction
%======================================

\section{Introduction}

This paper grows out of the attempt of studying moduli spaces of semistable sheaves on K3 surfaces and 
holomorphic symplectic manifolds from a new point of view. 

Holomorphic symplectic manifolds are complex manifolds with nowhere degenerate holomorphic 2-forms. 
They have very rich geometry and beautiful properties, mainly due to the interaction of two structures on
the second cohomology group, namely, the weight 2 Hodge decomposition and the Beauville-Bogomolov pairing. 
For example, they have unobstructed deformations \cite{Bogomolov-1978, Tian-1987, Todorov-1989, 
Ran-1992, Kawamata-1992}, local and global Torelli theorems \cite{Beauville-1983, Verbitsky-2009, Huybrechts-2011}. 
Furthermore, birational irreducible holomorphic symplectic manifolds are always deformation equivalent 
\cite{Huybrechts-2003}. Looking for examples of holomorphic symplectic manifolds is always a central
problem in this area. 

On the other hand, moduli spaces of semistable sheaves on a projective variety
\cite{Gieseker-1977, Maruyama-1977, Maruyama-1978}
has been a very popular research area in differential
geometry, algebraic geometry, gauge theory and theoretical physics since a long time ago. When the underlying
variety is a K3 surface, Mukai \cite{Mukai-1984} constructed a non-degenerate holomorphic 2-form on the smooth 
locus of the moduli space. Therefore, the smooth moduli spaces of sheaves on K3 surfaces provide a whole
series of examples of irreducible holomorphic symplectic manifolds. A similar results on abelian surfaces \cite{Beauville-1983}
yields another series of examples, which are are the so called generalized Kummer varieties. For quite a long time
these are the only known examples of irreducible symplectic varieties. A natural question to ask at this stage is:
since Mukai has showed the existence of a  holomorphic 2-form on the smooth locus of any singular moduli 
space of semistable sheaves on K3 surfaces, is there any way to turn these singular spaces into holomorphic 
symplectic manifolds? 

O'Grady's work \cite{OGrady-1999,OGrady-2003} partly answered this question. He studied such a 10-dimensional singular
moduli space, as well as a 6-dimensional moduli space of sheaves over an abelian surface, and constructed their symplectic 
resolutions. A comparison of topological invariants shows that they are two new examples of irreducible symplectic manifolds. 
Some work was done along this route, and eventually, Kaledin, Lehn and Sorger showed that, O'Grady's example was the only
one which could arise by desingularizing moduli spaces of sheaves on K3 surfaces \cite[Theorem 6.2]{Kaledin-2006}. 
This result is somehow a negative one which excludes many of the moduli spaces 
from the game, although they are very close to be symplectic manifolds. 

At the same time, people are trying to generalize the notion of symplectic manifolds to allow singularities. Beauville defined
the notion of symplectic singularities in \cite{Beauville-2000}. After that a lot of work was extensively done by many other 
people, such as \cite{Kaledin-2006a,Namikawa-2001,Namikawa-2001a}. In particular, in \cite[Theorem 6.2]{Kaledin-2006}, 
it is proved that all singular moduli spaces of sheaves on K3 surfaces are (singular) symplectic varieties in the sense of 
\cite{Beauville-2000}. 

Another reason why we should enlarge the notion of holomorphic symplectic manifolds to include all singular moduli 
spaces of sheaves on K3 surfaces roots in enumerative geometry and theoretical physics. In recent years, the study of
Donaldson-Thomas type invariants has grow into a large area involving many modern techniques in many different fields 
in algebraic geometry, such as deformation theory, stacks, derived categories and motives. Although a lot of work about 
Donaldson-Thomas type invariants on Calabi-Yau 3-folds is done, not so much is known on a K3 surface. On the other hand, 
Vafa and Witten in \cite{Vafa-1994} predicted from S-duality that the generating function of the Euler characteristics of
instanton moduli spaces on K3 surfaces has a modularity property. As a consequence, the Euler characteristics of singular
moduli spaces could be all determined by those of the smooth ones, with possible denominators 2 or 4. Mathematically 
there's no convincing definition of the Euler characteristics (which are presumbly Donaldson-Thomas type invariants) 
needed for this conjecture so far, and the contribution of the singularities of the moduli spaces to the denominators remains 
a mystery. 

In this paper, we are trying to generalize the notion of holomorphic symplectic manifolds into the stacky world. So that 
one has the possibility of dealing with all moduli spaces of semistable sheaves on K3 surfaces, when considered as Artin
stacks, in a uniform way, without the necessity of distinguishing them by the existence of symplectic resolutions. The 
role of the holomorphic symplectic form in the definition of the holomorphic symplectic manifolds, is to provide an isomorphism
of the tangent bundle, or rather the cotangent bundle, with its dual, such that the isomorphism is anti-symmetric. 
We generalize the category of manifolds to stacks, and replace the cotangent bundle by the cotangent complex. Therefore, 
motivated by the work on symmetric obstruction theories in
\cite{Behrend-2008}, in this paper we define the notion of a \textit{symplectic stack} as follows: 

\begin{definition}
A \textit{symplectic stack} is an algebraic stack, whose cotangent complex is a \textit{symplectic complex}, namely, 
a complex equipped with a  non-degenerate anti-symmetric bilinear pairing. 
\end{definition}
The precise definitions
can be found in Definition \ref{_definition_complex_}, \ref{_bad_definition_for_stack_} and \ref{_definition_stack_}. 

Besides the trivial examples of symplectic manifolds and quotients of symplectic manifolds by finite subgroups of symplectomorphisms, 
the major part of this paper is devoted to study the question, whether moduli stacks of semistable sheaves on K3 surfaces, when viewed
as a global quotient stack of the GIT-semistable locus of Quot scheme by the gauge group, are examples of symplectic stacks. 
The difficulty lies in the computation of their cotangent complexes. More precisely, we will prove the following results 
(see Theorem \ref{_quot_cotangent_complex_} and \ref{_stack_cotangent_complex_}): 

\begin{theorem}
The cotangent complexes of the GIT-semistable locus of the Quot scheme $Q$ and the moduli stack of semistable sheaves $\cM$ can 
be expressed explicitly by the universal quotient sequence on $Q$. More precisely, under the notations given at the end of this 
section, we have quasi-isomorphisms
\begin{eqnarray*}
\rpils\rhom(\cK,\cF)_0^\vee &\stackrel{\cong}{\lra}& \L_Q\\
\rpils\rhom(\cF,\cF)_0^\vee[-1] &\stackrel{\cong}{\lra}& q^*\L_{\cM}.
\end{eqnarray*}
\end{theorem}

And by using Serre duality, we finally get a positive conclusion to the above question, that is (see Theorem \ref{_moduli_symplectic_stack_}):

\begin{theorem}
The moduli stack $\cM$ of semistable sheaves on a K3 surfaces is a symplectic stack.
\end{theorem}

The main techniques in the computation were adopted from the paper \cite{Huybrechts-2010} of Huybrechts and Thomas
on the application of Atiyah class on deformation theory of complexes, and the paper \cite{Gillam-2011} of Gillam on the
application of reduced Atiyah class on the deformation theory of quotients. The paper is organized as follows: 

In section 2, 
we first of all briefly recall some properties of cotangent complexes which will be used later, then we turn to a short summary 
of classical Atiyah classes and reduced Atiyah classes, including their definitions and properties in the deformation-obstruction theory. 
We will also show that the reduced Atiyah class is a lift of the classical Atiyah class. 

Section 3 is mainly a technical point. Since we will eventually be interested in the moduli space of sheaves with fixed determinant, 
we have to remove the effect of the trace map. This section uses techniques in derived categories to create ``traceless version" of 
all complexes involved in the following sections. 

Section 4 contains the first half of the central computation, which is on the cotangent complex of the GIT-semistable locus of the Quot
scheme. Following \cite{Huybrechts-2010, Gillam-2011}, we use the reduced Atiyah class to establish a morphism from a complex
constructed only from the universal family on the Quot scheme, to the cotangent complex of Quot scheme. Then we show that this
morphism induces isomorphisms on all cohomology groups. 

Section 5 provides the other half of the central computation. We use the transitivity property of the cotangent complex, together 
with the cotangent complex of the Quot scheme computed in previous section to obtain the cotangent complex of the quotient
stack. The commutativity of the diagram \ref{_traceless_commute_} is the major obstacle that we have to overcome in this section. 

In section 6, we take the definition of symmetric obstruction theories in \cite{Behrend-2008} as a model, and
formally introduce the notion of a symplectic stack. As an application of the computations in previous sections, 
we show that the moduli stacks we studied in previous sections are indeed examples of symplectic stacks. 

\begin{notations}
Throughout this paper, $X$ will always be a projective K3 surface, and $H$ is an ample line bundle on $X$, which is used to determine
the stability of sheaves in Gieseker's sense. We always use $Q$ for the GIT-semistable locus of the Grothendieck's Quot scheme used in the 
GIT construction of the moduli space. For simplicity, sometimes we will omit the words ``GIT-semistable locus", but we will never take
the unstable locus into consideration. We denote the two projections from $Q\times X$ by 
\begin{equation*}
\xymatrix{
 & Q\times X \ar_{\pi}[ld] \ar^{\pi_X}[rd] & \\
Q & & X.
}
\end{equation*}
The gauge group $PGL(N)$ in the GIT construction will be denoted by $G$, and the global quotient stack $[Q/G]$ will be denoted by $\cM$. 
We use
$$q: Q\lra \cM$$
for the structure morphism from $Q$ to the global quotient stack $\cM$. 
We always assume that there is at least one stable quotient sheaf. Since the stability is an open condition, 
we denote the open dense subscheme of $Q$ over which
the quotient sheaf is stable by $Q^s$, and the corresponding image of $Q^s$ under $q$  by $\cM^s$. 

We also fix the universal quotient sequence 
\begin{equation}\label{_universal_quotient_}
0 \lra \cK \lra \cE \lra \cF \lra 0
\end{equation}
on the GIT-semistable locus of the Quot scheme, or more precisely, on $Q\times X$. 
Here we should note that $\cE$ is obtained by pulling back a vector bundle on $X$ via the projection $\pi_X$, therefore is
a trivial family over $Q$. 
\end{notations}

\begin{acknowledgements}
The author owes a great debt of gratitude to Professor Jun Li, who suggested and encouraged the author to work towards the new notion 
of symplectic stacks, and to Professor Daniel Huybrechts, who generously shared his idea of formulating the concept using 
cotangent complex, and had numerous helpful discussions with the author on this topic. The author would also like to thank 
Professor Lehn, for his interest and many helpful discussions about this work, and most importantly, for his support 
via SFB/TR 45 during 
the time when most of this paper was written. The author also wants to thank Professor Dan Edidin, Professor Zhenbo Qin, 
Professor Ravi Vakil, Professor Justin Sawon, Professor S\"onke Rollenske, Jason Lo and Timo Sch\"urg for their interests and
discussions. The author also wants to thank Max-Planck-Institute for Mathematics in Bonn for their support during the final 
stage of this work.
\end{acknowledgements}

%======================================
%       Section Two: Atiyah Classes
%======================================

\section{Cotangent Complexes and Atiyah Classes}

\subsection{Cotangent Complexes}

We first of all recall some properties of cotangent complexes, which will be important for our later discussions. 
The classical reference for cotangent complexes is \cite{Illusie-1971}. For cotangent complexes of stacks, one 
can see \cite[Chapter 17]{Laumon-2000} and \cite{Olsson-2007}. 

\begin{lemma}\cite{Illusie-1971,Laumon-2000}
% For the case of local complete intersections, see Illusie Proposition 3.2.6.
% The case of stacks, maybe in LMB or Olsson? 
The cotangent complex of $\mathcal{X}$ is an object in the derived category $D^b(\mathcal{X})$, which
\begin{enumerate}
\item is quasi-isomorphic to a single locally free sheaf in degree $0$ if $\mathcal{X}$ is a smooth scheme;
\item has perfect amplitude in $[-1,0]$ and is quasi-isomorphic to a single sheaf in degree $0$ if $\mathcal{X}$ is 
a scheme of locally complete intersection;
\item has perfect amplitude in $(-\infty, 0]$ if $\mathcal{X}$ is a scheme or a Deligne-Mumford stack;
\item has perfect amplitude in $(-\infty, 1]$ if $\mathcal{X}$ is an Artin stack. 
\end{enumerate}
\qed
\end{lemma}

The computation of the cotangent complex is in general very difficult, however the following functorial property
turns out to be very helpful in some cases. 

\begin{proposition}\cite{Illusie-1971,Laumon-2000}\label{_cotangent_functorial_}
% In Illusie, it's proposition 2.1.2.
Let $$\mathcal{X}\stackrel{f}{\lra}\mathcal{Y}\stackrel{g}{\lra}\mathcal{Z}$$ be two morphisms of schemes
or stacks, then we have the following exact triangle in $D^b(\mathcal{X})$:
\begin{equation}
f^*\L_{\cY/\cZ} \lra \L_{\cX/\cZ} \lra \L_{\cX/\cY}. 
\end{equation}
Furthermore, this exact triangle is functorial. Namely, if we have a commutative diagram
\begin{equation*}
\xymatrix{
\cX \ar^f[r] \ar_u[d] & \cY \ar[r] \ar[d] & \cZ \ar[d]\\
\cX' \ar^{f'}[r] & \cY' \ar[r] & \cZ',
}
\end{equation*}
then there are morphisms between two exact triangles (vertical arrows in the following diagram) 
which makes the diagram commute
\begin{equation*}
\xymatrix{
u^*f'^*\L_{\cY'/\cZ'} \ar[r] \ar[d] & u^*\L_{\cX'/\cZ'} \ar[r] \ar[d] & u^*\L_{\cX'/\cY'} \ar[d]\\
f^*\L_{\cY/\cZ} \ar[r] & \L_{\cX/\cZ} \ar[r] & \L_{\cX/\cY}
}
\end{equation*}
\qed
\end{proposition}

The other important property which is helpful in understanding the cotangent complex of a scheme is

\begin{lemma}\cite{Illusie-1971}
Let $X$ be a scheme, then $$H^0(\L_X)=\Omega_X,$$ the cotangent sheaf of $X$. In particular, if $X$ 
is of local complete intersection (or even smooth), we have a quasi-isomorphism $$\L_X \stackrel{\cong}{\lra}
\Omega_X.$$
\end{lemma}

\begin{proof}
See \cite[Proposition 1.2.4.2]{Illusie-1971}.
\end{proof}

The reason why cotangent complexes are important is that they play a central role in deformation theory, which is the
whole essence of \cite{Illusie-1971,Illusie-1972}. For our purpose, we need two properties 
of cotangent complexes concerning deformation theory. 

Let $X$ be any scheme, and $I$ be any coherent $\cO_X$-module. We use the notion 
$$X[I]=\spec_X(\cO_X\oplus I)$$
for the trivial square zero extension of $X$ by $I$. In other words, $X[I]$ can be viewed as a first order 
deformation of $X$, and we denote the corresponding natural inclusion by $\iota_I: X \lra X[I]$. 
A retract of $X[I]$ is defined to be a morphism $r: X[I] \lra X$ such that
the composition $r\circ \iota_I=\textrm{id}_X$.

\begin{proposition}\cite{Illusie-1971}\label{_cotangent_retracts_}
Under the above notations, all retracts from $X[I]$ to $X$ are parametrized by $\hom(\L_X, I)$. Or in other words, 
the automorphism group of $X[I]$ is given by $\hom(\L_X, I)$. 
\end{proposition}

\begin{proof}
This is part of \cite[Theorem III.2.1.7]{Illusie-1971}. 
Since $X$ is a scheme, by the above lemmas, we have 
$$\hom(\L_X, I) = \hom(\Omega_X, I) = \der(\cO_X, I).$$
However every retract $r: X[I] \lra X$ corresponds to a splitting $\varphi$ of the exact sequence
\begin{displaymath}
\xymatrix{
0 \ar[r] & I \ar[r] & \cO_X\oplus I \ar[r] & \cO_X \ar[r] \ar@/^/[l]^{\varphi} & 0,
}
\end{displaymath}
which is an algebra homomorphism, therefore corresponds to a derivation into $I$. 
\end{proof}

A priori, $X[I]$ may not be the only first order thickening of $X$ by the ideal $I$. The following 
proposition gives a parameter space for all such thickenings. It's an application of the
so-called ``Fundamental Theorem of Cotangent Complex" in \cite{Illusie-1971}.

\begin{proposition}\cite{Illusie-1971}\label{_cotangent_thickenings_}
Under the above notations, all isomorphism classes of first order thickenings of $X$ by the ideal $I$ are 
parametrized by $\ext^1(\L_X, I)$. 
\end{proposition}

\begin{proof}
See \cite[Theorem III.1.2.3, III.1.2.7]{Illusie-1971}.
\end{proof}

The class corresponding the a particular thickening of $X$ is the so-called \textit{Kodaira-Spencer class}.
Obviously, the zero element in $\ext^1(\L_X, I)$ represents the trivial thickening $X[I]$. 

So much general theory of cotangent complexes. Now we want to show that, the GIT-semistable locus of Quot scheme, 
which is, by abuse of notation, denoted by $Q$, 
although is 
in general singular due to the existence of strictly semistable sheaves, 
in fact has a simple cotangent complex, namely, quasi-isomorphic to a single sheaf concentrated 
in degree 0. 

\begin{lemma}\label{_quot_lci_}
The GIT-semistable locus of the Quot scheme $Q$ in the GIT construction of the moduli space of semistable sheaves 
on a K3 surface is a local complete intersection. In particular, the cotangent complex $\L_Q$ has perfect amplitude
in $[-1,0]$, and is quasi-isomorphic to the cotangent sheaf $\Omega_Q$. 
\end{lemma}

\begin{proof}
By \cite[Proposition 2.2.8]{Huybrechts-2010}, 
%\red{(or something similar? since we need a reduced obstruction theory,)} 
at every closed point 
$q \in Q$, which is by the above notation represented by a quotient $$0 \lra K \lra E \lra F \lra 0,$$ 
there is an inequality concerning the dimension of $Q$ at the point
$$\dim\hom(K,F) \geq \dim_qQ \geq \dim\hom(K,F) - \dim\ext^1(K,F)_0,$$ where
$\ext^1(K,F)_0$ is the kernel of the composition map $$\ext^1(K,F) \stackrel{\cong}{\lra} \ext^2(F,F) \stackrel {tr}{\lra}H^2(\cO_X) = \C.$$
And the Quot scheme $Q$ is a local complete intersection if and only if the second equality holds at every closed point $q\in Q$. 

By \cite[Theorem 4.4]{Kaledin-2006} (see also \cite[Theorem 3.18]{Yoshioka-2003}), we know that the GIT-semistable locus of the Quot scheme $Q$
is irreducible. Therefore, 
$\dim_qQ$ is constant on the only connected component of $Q$. Furthermore, it's easy 
to check that for every $i\geq 2$, we have $$\ext^i(K,F)=0, $$ which implies $$\dim\hom(K,F) - \dim\ext^1(K,F)_0=\chi(K,F)+1$$ which is a topological 
number only depending on the Chern classes of $F$, hence is also constant. Therefore it suffices to check the equality of both sides at one 
closed point of $Q$. 

However by the assumption, there exists at least one point in the Quot scheme $Q$ which is represented by a stable quotient sheaf $F$. At such
a point the obstruction space $\ext^1(K,F)_0$ vanishes, therefore both equalities hold at the same time. By the above discussion we conclude that
the Quot scheme $Q$ is a local complete intersection. 
\end{proof}

\subsection{Classical Atiyah Classes}

In \cite{Illusie-1971}, the Atiyah class  were defined in two different ways. We follow the second approach using 
the exact sequence of principal parts, which itself was defined in \cite[III.1.2.6]{Illusie-1971}. 

Let $A\lra B$ be a ring homomorphism, then we have an exact sequence
$$0 \lra I \lra B\otimes_A B \lra B \lra 0$$ which splits by either of the ring homomorphisms
$$j_1, j_2: B\lra B\otimes_A B$$ where $$j_1(x)=x\otimes 1, \ \ j_2(x)=1\otimes x.$$ After dividing by $I^2$ we obtained
$$0 \lra I/I^2 \lra B\otimes_A B/I^2 \lra B \lra 0$$ which we denoted by $$0 \lra \Omega_{B/A} \lra P^1_{B/A} \lra B \lra 0.$$
Note that $P^1_{B/A}$ is a $B$-$B$-bimodule. 
Let $M$ be a $B$-module, we tensor the above exact sequence by $M$ from right side and obtain
$$0 \lra \Omega_{B/A}\otimes_B M \lra P^1_{B/A}\otimes_B M \lra M \lra 0,$$ which define a class in $\ext^1_B(M, M\otimes \Omega_{B/A})$. 
This class is called the \textit{Atiyah class of $M$}. 

In our settings, we let $A$ be $\cO_X$ and $B$ be $\cO_{Q\times X}$. For any sheaf $\cF$ on
$Q\times X$, we obtained the Atiyah class of $\cF$, which we denote by $At(\cF)$. Note that, a priori, what we defined above is only
a truncation of the full Atiyah class. For the definition of the full Atiyah class, we need to replace $B$ by a simplicial resolution in the sequence 
of principal parts. However, because of Lemma \ref{_quot_lci_}, there's no difference in this case. Note that
$$\Omega_{Q\times X/X}=\L_{Q\times X/X}=\pi^*\L_Q,$$ so we actually have defined the (full) Atiyah class 
$$At(\cF)\in\ext^1_{Q\times X}(\cF, \cF\otimes\pi^*\L_Q).$$ 

Now we study the deformation properties of the Atiyah class. We have already seen from Lemma \ref{_cotangent_retracts_}
that, for any coherent sheaf $I$ on $Q$, the space $\hom_Q(\L_Q, I)$ parametrizes all retracts $\iota: Q[I]\lra Q$. On the other hand, 
from classical sheaf deformation theory, we also know that $\ext^1_{Q\times X}(\cF, \cF \otimes \pi^*I)$ parametrizes all flat deformations
of $\cF$ from $Q$ to $Q[I]$ (see for example \cite[Lemma 3.4]{Thomas-2000}). The Atiyah class gives gives a nice relation of the
two spaces as follows:

\begin{proposition}
Let $I$ be any coherent $\cO_Q$-module and $\cF$ be a coherent sheaf of $\cO_{Q\times X}$-module which is flat over $Q$. 
Let $At(\cF)$ be its Atiyah class defined as above. Then the map
$$At(\cF)\cup (\cF\otimes - ) : \hom_Q(\L_Q, I) \lra \ext^1_{Q\times X}(\cF, \cF\otimes \pi^*I)$$ given by precomposing with the Atiyah class of $\cF$
can be interpreted as 
$$\{\textrm{retracts of }\iota_I\} \lra \{\textrm{flat deformations of the }\cF\textrm{ over }Q[I]\},$$
where the arrow is given by pulling back the sheaf $\cF$ via the chosen retract. 
\end{proposition}

The proof of the proposition is very straightforward and is just a matter of unwinding the definitions. However, to the best 
of my knowledge, it doesn't seem to appear anywhere in this form. So we include a proof here. 

\begin{proof}
From Lemma \ref{_cotangent_retracts_}, we actually know that, for any $$u\in\hom_Q(\L_Q, I)=\hom_Q(\Omega_Q, I),$$
the corresponding retract is given by the splitting of the second row in the following diagram via the arrow $(\textrm{id}, u\circ d_Q)$:
\begin{equation*}
\xymatrix{
0 \ar[r] & \Omega_Q \ar[r] \ar^u[d] & \cO_Q\oplus\Omega_Q \ar[r] \ar[d] & \cO_Q \ar[r] \ar@{=}[d] \ar@/_/_{(\textrm{id}, d_Q)}[l] & 0\\
0 \ar[r] & I \ar[r] & \cO_Q\oplus I \ar[r] & \cO_Q \ar[r] \ar@/^/^{(\textrm{id}, u\circ d_Q)}[l] & 0,
}
\end{equation*}
where $d_Q$ is the universal derivation defined by $$d_Q(x)=1\otimes x-x\otimes 1$$ for every $x\in\cO_Q$. 

Note that $\cO_Q\oplus\Omega_Q$ is exactly the principal part $P^1_Q$ together with the left $\cO_Q$-module structure, 
while the splitting $(\textit{id}, d_Q)$ is exactly the right $\cO_Q$-module structure. Similarly, $\cO_Q\oplus I$ can also be
identified with $\cO_{Q[I]}$ such that the splitting $(\textrm{id}, u\circ d_Q)$ correspond exactly to the retract. 

We pullback the diagram to $Q\times X$ and tensor every term with the sheaf $\cF$, using the splittings of both rows, then we get
\begin{equation*}
\xymatrix{
0 \ar[r] & \pi^*\Omega_Q\otimes\cF \ar[r] \ar[d] & P^1_{Q\times X/X}\otimes\cF \ar[r] \ar[d] & \cF \ar[r] \ar@{=}[d] & 0\\
0 \ar[r] & \pi^*I\otimes\cF \ar[r] & \pi^*Q[I]\otimes\cF \ar[r] & \cF \ar[r] & 0,
}
\end{equation*}

From the above construction we see exactly that the second row is the class $At(\cF)\cup (\cF\otimes u)$, which finishes the proof.
\end{proof}

The above proposition concerns the relation of the Atiyah class with deformations of $\cF$ over the trivial first order deformation $Q[I]$
of $Q$. Next proposition relates the Atiyah class with obstructions of deforming $\cF$ to an arbitrary first order deformation of $Q$. 
More precisely, we have seen from Lemma \ref{_cotangent_thickenings_} that $\ext^1(\L_Q, I)$ parametrizes all isomorphism
classes of first order thickenings of $Q$ by the ideal sheaf $I$. And classical deformation theory tells us that the obstruction of 
deforming $\cF$ to any first order extension of the base lies in the space $\ext^2(\cF, \cF\otimes\pi^*I)$
(see for example \cite[Proposition 3.13]{Thomas-2000}). Via the Atiyah class, we
can make a precise formation of their relation:

\begin{proposition}\cite{Illusie-1971}
Let $I$ be any coherent $\cO_Q$-module and $\cF$ be a coherent sheaf of $\cO_{Q\times X}$-module which is flat over $Q$. 
Let $At(\cF)$ be its Atiyah class defined as above. The map
$$At(\cF) \cup (\cF\otimes - ) : \ext^1_Q(\L_Q, I) \lra \ext^2_{Q\times X}(\cF,\cF\otimes\pi^*I)$$
given by precomposing with the Atiyah class can be interpreted as 
\begin{multline*}
\{\textrm{thickenings }Q' \textrm{ of }Q \textrm{ by }\iota_I\} \lra 
\{\textrm{obstructions to the existence }\\
\textrm{of flat deformations of } \cF \textrm{ over }Q'\}.
\end{multline*}
\end{proposition}

\begin{proof}
This proposition is under the general principle of ``the composition of Atiyah class and Kodaira-Spencer class is the obstruction". 
The proof can be found in \cite[Proposition IV.3.1.8]{Illusie-1971}.
\end{proof}

\subsection{Reduced Atiyah Classes}\label{_section_reduced_atiyah_}

Now we turn to the reduced Atiyah class, which was defined and extensively studied in \cite{Gillam-2011}. Two definitions
were given, one using graded cotangent complex, the other using more classical language. We follow the second approach 
in \cite{Gillam-2011} and give a brief definition of the reduced Atiyah class in our context, under the additional property 
that $Q$ is a local complete intersection, just to avoid any simplicial resolution of algebras. 

Recall that we have the short exact sequence of sheaves on $Q\times X$ given by \ref{_universal_quotient_}, where 
$$\cE=\pi^*_XE_0$$ is a trivial family of vector bundles over $Q$. We have the following commutative diagram
with all rows and columns exact: 

\begin{equation*}
\xymatrix{
 & 0 \ar[d] & 0 \ar[d] & 0 \ar[d] & \\
0 \ar[r] & \cK\otimes\pi^*\Omega_Q \ar[d] \ar[r] & \cE\otimes\pi^*\Omega_Q \ar[d] \ar[r] & \cE\otimes\pi^*\Omega_Q \ar[d] \ar[r] & 0\\
0 \ar[r] & P^1(\cK) \ar[d] \ar[r] & P^1(\cE) \ar[d] \ar[r] & P^1(\cF) \ar[d] \ar[r] & 0\\
0 \ar[r] & \cK \ar[d] \ar[r] & \cE \ar[d] \ar[r] \ar@/_/_{\sigma}[u] & \cF \ar[d] \ar[r] & 0\\
 & 0 & 0 & 0. & 
}
\end{equation*}

The exactness of the three columns are trivial, because they are all exact sequence of principal parts. The exactness of 
the third row is part of given data. The exactness of the first row is by the flatness of $\cF$ and exactness of the middle row
comes from that of the other two rows. 

As observed in \cite{Gillam-2011}, the middle column naturally splits, due to the fact that $\cE$ is a trivial family over $Q$. 
The reversed arrow $\sigma$ in the above diagram is chosen as follows: we have $$\cE=\cO_Q\otimes_{\C}E_0$$ and
$$P^1(\cE)=(\cO_{Q\times Q}/{I^2})\otimes_{\C}E_0=(\cO_Q\otimes_{\C}\cO_Q/I^2)\otimes_{\C}E_0,$$
where $I$ is the ideal sheaf of the diagonal in $Q\times Q$. 
Then we define
\begin{eqnarray*}
\sigma:\quad\quad \cE &\lra& P^1(\cE) \\
a\otimes e &\longmapsto& a\otimes 1\otimes e. 
\end{eqnarray*}
It's obvious that $\sigma$ is indeed a splitting. 

After all of the preparation, we define the reduced Atiyah class $at\in\hom(\cK, \cF\otimes\pi^*\L_Q)$ as the composition
of the following arrows from the above diagram, starting from $\cK$: 
\begin{equation}\label{_reduced_construction_}
\xymatrix{
 & & & F\otimes\pi^*\L_Q \ar[d]\\
 & & P^1(\cE) \ar[r] & P^1(\cF)\\
 & \cK \ar[r] & \cE \ar_{\sigma}[u], & 
}
\end{equation}
where the downward arrow in the right column means that, everything in $P^1(\cF)$ which comes from $\cK$ via 
the composition of the other three arrows is in the image of this downward arrow, therefore can be uniquely lifted
to $\cF\otimes\pi^*\L_Q$. The reason is that if we maps it further down to $\cF$ as in the above diagram, we get
the zero section. Hence the composition of the four arrows is well-defined. 

The reduced Atiyah class behaves compatibly with the classical Atiyah class. In fact, it is a lifting of the classical Atiyah class, 
as we can see from next property, 

\begin{proposition}\label{_atiyah_compatibility_}
Let $at\in \hom(\cK,\cF\otimes \pi^*\L_Q)$ be the reduced Atiyah class, and $e\in \ext^1(\cF,\cK)$ be the extension class represented 
by the universal quotient sequence \ref{_universal_quotient_} on $Q$. Then
\begin{enumerate}
\item The composition $$(at[1]) \circ e: \cF \lra \cK[1] \lra \cF\otimes \pi^*\L_Q[1]$$ is the classical Atiyah class $At(\cF)$;
\item The composition $$(e\otimes \pi^*\L_Q) \circ at: \cK \lra \cF\otimes \pi^*\L_Q \lra \cK\otimes \pi^*\L_Q[1]$$ is the classical Atiyah class $At(\cK)$.
\end{enumerate}
\end{proposition}

\begin{proof}
We will only use the first half of the proposition, so only this part will be proved in details. However, the proof of the second part of the
proposition is completely parallel to that of the first part. 

To prove the first part of the proposition, it suffices to show that
\begin{equation*}
\xymatrix{
0 \ar[r] & \cK \ar[r] \ar[d] & \cE \ar[r] \ar[d] & \cF \ar[r] \ar@{=}[d] & 0\\
0 \ar[r] & \cF\otimes\pi^*\L_Q \ar[r] & P^1(\cF) \ar[r] & \cF \ar[r] & 0
}
\end{equation*}
is a pushout diagram, where the first row is the universal family \ref{_universal_quotient_}, while the second row is the principal part
sequence for $\cF$. By the construction of the pushout, in fact we just need to show that 
\begin{equation}\label{_push_out_}
0 \lra \cK \stackrel{\varphi_1}{\lra} \cE\oplus (\cF\otimes\pi^*\L_Q) \stackrel{\varphi_2}{\lra} P^1(\cF) \lra 0
\end{equation}
is exact, where the map $\varphi_1$ is the pair of the first arrow in diagram \ref{_reduced_construction_} and the negation of the reduced
Atiyah class $-at$, and the map $\varphi_2$ is the sum of composition of the middle two arrows in \ref{_reduced_construction_} and 
the downward arrow. 

To verify this claim, we observe that
\begin{itemize}
\item $\varphi_1$ is injective, which is obvious because the first component is injective;
\item $\varphi_2$ is surjective. In fact, $\im(\varphi_2)$ is a submodule of $P^1(\cF)$, and obviously $\cF\otimes\pi^*\L_Q$ lies in 
$\im(\varphi_2)$. Furthermore, $\im(\varphi_2)/(\cF\otimes\pi^*\L_Q)=\cF$ because the image of $\cE$ hits everything in $\cF$;
\item $\im(\varphi_1)\subset\ker(\varphi_2)$, which is due to the construction of the reduced Atiyah class;
\item $\ker(\varphi_2)\subset\im(\varphi_1)$. In fact, if $(e, f')\in \cE\oplus (\cF\otimes\pi^*\L_Q)$ satisfies $\alpha_2(e,f')=0$, then $e$
comes from a certain $k\in\cK$, because its image in $\cF$ is the negation of the image of $f'$in $\cF$, which is $0$. Then it's clear
that $\varphi_1(k)=(e,f')$. 
\end{itemize}
The above observations finish the proof of the exact sequence \ref{_push_out_}. 
\end{proof}

Similar to the discussion of the classical Atiyah class, we will also need to use some deformation interpretations of the reduced Atiyah class. 
We know from the deformation theory of quotients (for example, in Chapter 2 of \cite{Huybrechts-2010b}) that, 
for any coherent $\cO_Q$-module $I$, the space $\hom(\cK, \cF\otimes \pi^*I)$ parametrizes all first order
flat deformations of the universal quotient \ref{_universal_quotient_} to the square 0 extension $Q[I]$. 
We also know that $\ext^1_{Q\times X}(\cK,\cF\otimes\pi^*I)$ contains 
all obstruction classes of lifting the quotient \ref{_universal_quotient_} to the first order.

The following two propositions from \cite{Gillam-2011} concerning the relation between the reduced Atiyah class and the deformation theory.
The statements are parallel to similar results in previous section about classical Atiyah class. The first one relates the reduced Atiyah 
class with deformation of quotients on the trivial square free deformation of the base:

\begin{proposition}\cite{Gillam-2011}\label{_deformation_reduced_atiyah_}
Let $I$ be any coherent $\cO_Q$-module and $at$ be the reduced Atiyah class of the quotient \ref{_universal_quotient_}. The map
$$at \cup (\cF\otimes - ) : \hom_Q(\L_Q, I) \lra \hom_{Q\times X}(\cK,\cF\otimes\pi^*I)$$
given by precomposing with the reduced Atiyah class can be interpreted as 
$$\{\textrm{retracts of }\iota_I\} \lra \{\textrm{flat deformations of the quotient \ref{_universal_quotient_} over }Q[I]\},$$
where the arrow is given by pullback. 
\end{proposition}

\begin{proof}
See the proof in \cite[Lemma 3.2]{Gillam-2011}. 
\end{proof}

Next proposition concerns the relation between the Atiyah class and the obstruction of 
deformation of the quotient map, which is also under the essence of ``the product of Atiyah class and Kodaira-Spencer class
is the obstruction class". 

\begin{proposition}[\cite{Gillam-2011}]
Let $I$ be any coherent $\cO_Q$-module and $at$ be the reduced Atiyah class of the quotient \ref{_universal_quotient_}. The map
$$at \cup (\cF \otimes - ) : \ext^1_Q(\L_Q, I) \lra \ext^1_{Q\times X}(\cK,\cF\otimes\pi^*I)$$
given by precomposing with the reduced Atiyah class can be interpreted as 
\begin{multline*}
\{\textrm{thickenings }Q' \textrm{ of }Q \textrm{ by }\iota_I\} \lra 
\{\textrm{obstructions to the existence }\\
\textrm{of flat deformations of the quotient \ref{_universal_quotient_} over }Q'\}.
\end{multline*}
\end{proposition}

\begin{proof}
See \cite[Lemma 1.13]{Gillam-2011}.
\end{proof}

%=================================================
%       Section Three: Trace Maps and Trace-free parts of Complexes
%=================================================

\section{Trace Maps and Trace-free parts of Complexes}

On the Quot scheme $Q$, we apply the derived functor $\rpils\crhom(-, \cF)$ to the exact sequence 
\ref{_universal_quotient_} and get an exact triangle of complexes
\begin{equation}\label{_original_triangle_}
\rpils\crhom(\cK,\cF)[-1] \lra \rpils\crhom(\cF,\cF) \lra \rpils\crhom(\cE,\cF).
\end{equation}
In this section we will construct the ``traceless" version of all
the three complexes. 

First of all it's easy to see that the trace map $tr: \crhom(F,F) \lra \cO_{Q\times X}$ splits by a reverse map of scaling. Therefore we have 
$$\crhom(F,F)=\crhom(F,F)_0 \oplus \cO_{Q\times X},$$
where the first summand is the kernel of the above trace map. The splitting leads to 
$$\rpils\crhom(F,F)=\rpils\crhom(F,F)_0 \oplus \rpils\cO_{Q\times X}.$$

However, the second component above can be further decomposed into 
two direct summands as
\begin{eqnarray*}
\rpils\cO_{Q\times X} &=& \rpils\pi_X^*\cO_X \\
&=& R\Gamma(\cO_X) \otimes \cO_Q \\
&=& (H^0(\cO_X) \oplus H^2(\cO_X)[-2]) \otimes \cO_Q \\
&=& \pi_*\cO_{Q\times X} \oplus \rpils\cO_{Q\times X}[-2]
\end{eqnarray*}

Therefore we have a decomposition of the middle complex of \ref{_original_triangle_}
\begin{equation}\label{_middle_decomposition_}
\rpils\crhom(F,F)=\rpils\crhom(F,F)_0 \oplus \pi_*\cO_{Q\times X} \oplus \textrm{R}^2\pi_*\cO_{Q\times X}[-2],
\end{equation}
in which we also keep in mind that 
\begin{eqnarray*}
\pi_*\cO_{Q\times X}&=&\cO_Q\\
\textrm{R}^2\pi_*\cO_{Q\times X}&=&\cO_Q
\end{eqnarray*}

We combine the equation \ref{_original_triangle_} and the obvious morphisms of 
embeddings and splittings from the equation \ref{_middle_decomposition_} and get two morphism
\begin{eqnarray*}
&\alpha:& \pi_*\cO_{Q\times X} \lra \rpils\crhom(\cE,\cF);\\
&\beta:& \rpils\crhom(\cK,\cF)[-1] \lra \textrm{R}^2\pi_*\cO_{Q\times X}[-2].
\end{eqnarray*}

Now we define the ``traceless" version of the other two complexes by completing the exact triangles. More precisely, 
we define
\begin{eqnarray*}
\rpils\crhom(\cE, \cF)_0&=&\textrm{Cone}(\alpha)\\
\rpils\crhom(\cK, \cF)_0&=&\textrm{Cone}(\beta).
\end{eqnarray*}

Then we have the following
\begin{proposition}
We naturally get an exact triangle 
\begin{equation}\label{_traceless_triangle_}
\rpils\crhom(\cK,\cF)_0[-1] \lra \rpils\crhom(\cF,\cF)_0 \lra \rpils\crhom(\cE,\cF)_0.
\end{equation}
\end{proposition}

\begin{proof}
We observe two exact triangles from the above cone construction:
\begin{equation}\label{_EF_traceless_}
\pi_*\cO_{Q\times X} \lra \rpils\crhom(\cE,\cF) \lra \rpils\crhom(\cE,\cF)_0,
\end{equation}
and 
\begin{equation}\label{_KF_traceless_}
\rpils\crhom(\cK,\cF)_0[-1] \lra \rpils\crhom(\cK,\cF)[-1] \lra \textrm{R}^2\pi_*\cO_{Q\times X}[-2].
\end{equation}
Then this proposition is just a direct consequence of next lemma. 
\end{proof}

\begin{lemma}
Let $$A \lra B \lra C$$ be an exact triangle in a triangulated category, where $B=B_1 \oplus B_2.$
\begin{enumerate}
\item If we complete the natural morphism 
$A \lra B_2$ into an exact triangle $$A_0 \lra A \lra B_2,$$ then we get a new exact triangle $$A_0 \lra B_1 \lra C;$$
\item If we complete the natural morphism 
$B_1 \lra C$ into an exact triangle $$B_1 \lra C \lra C_0,$$ then we get a new exact triangle $$A \lra B_2 \lra C_0.$$
\end{enumerate}
\end{lemma}

\begin{proof}
They are both applications of octohedral axiom of triangulated categories. 
\end{proof}

Next we analyze the fiberwise behaviour of the exact triangle \ref{_traceless_triangle_} and the corresponding 
cohomology groups. Let $p\in Q$ be a closed point. Let $X_p$ be the corresponding fiber in the product $Q\times X$, 
and 
\begin{equation}\label{_quotient_point_}
0 \lra K_p \lra E_p \lra F_p \lra 0
\end{equation}
be the corresponding quotient represented by $p$. 
Then the restriction of the decomposition \ref{_middle_decomposition_} becomes 
$$\rhom(F_p, F_p)=\rhom(F_p, F_p)_0 \oplus H^0(\cO_{X_p}) \oplus H^2(\cO_{X_p})[-2].$$ 

When we restrict the exact triangle \ref{_EF_traceless_} to the closed point $p$, we get the exact triangle
$$H^0(\cO_{X_p}) \lra \rhom(E_p, F_p) \lra \rhom(E_p, F_p)_0.$$ 
When we consider the corresponding long exact sequence of the cohomology groups, we realize that the complexes 
$\rhom(E_p, F_p)$ and $\rhom(E_p, F_p)_0$ actually computes the same cohomology groups except in degree 0, 
where we get an exact sequence $$0 \lra H^0(\cO_{X_p}) \stackrel{\alpha_p}{\lra} \hom(E_p, F_p) \lra \hom(E_p, F_p)_0 \lra 0.$$
From the above construction we see that the arrow $\alpha_p$ factor through $\hom(F_p, F_p)$ by a scalar map
$H^0(\cO_{X_p}) \lra \hom(F_p, F_p)$ and a natural map induced by the quotient \ref{_quotient_point_}, therefore 
$\hom(E_p, F_p)_0$ is obtained by ``removing" the 1-dimensional vector space generated by the map in the 
quotient \ref{_quotient_point_}. 

Similarly, we can analyze the exact triangle \ref{_KF_traceless_} and get a parallel conclusion. Summarizing the discussion 
we obtain the following lemma

\begin{lemma}\label{_fiber_cohomology_}
We have the following pointwise behaviour of the ``traceless" complexes $\rpils\crhom(\cE,\cF)_0$ and $\rpils\crhom(\cK,\cF)_0$:
\begin{enumerate}
\item
The restriction of the complex $\rpils\crhom(\cE,\cF)_0$ to any closed point $p\in Q$ computes the cohomology groups
\begin{equation*}
\ext^i(E_p, F_p)_0=
\begin{cases}
\ext^i(E_p, F_p) & \mbox{if } i\neq 0;\\
\cok (\alpha_p) & \mbox{if } i=0,
\end{cases}
\end{equation*}
where $\alpha_p$ is the composition the scalar and the natural map induced by \ref{_quotient_point_} 
$H^0(\cO_{X_p}) \lra \hom(F_p, F_p) \lra \hom(E_p, F_p)$. 
\item
The restriction of the complex $\rpils\crhom(\cK,\cF)_0$ to any closed point $p\in Q$ computes the cohomology groups
\begin{equation*}
\ext^i(K_p, F_p)_0=
\begin{cases}
\ext^i(K_p, F_p) & \mbox{if } i\neq 1;\\
\ker (\beta_p) & \mbox{if } i=1,
\end{cases}
\end{equation*}
where $\beta_p$ is the composition of the natural map induced by \ref{_quotient_point_} and the trace map
$\ext^1(K_p, F_p) \lra \ext^2(F_p, F_p) \lra H^2(\cO_{X_p})$. 
\end{enumerate}
\qed
\end{lemma}

%=================================================
%       Section Four: Cotangent Complex of the Quot Scheme
%=================================================

\section{Cotangent Complex of the Quot Scheme}

The goal of this section is to compute the cotangent complex of the GIT-semistable locus of the Quot scheme $Q$. 

\begin{lemma}\label{_reduced_atiyah_map_}
The reduced Atiyah class $at$ induces a morphism from $\rpils\crhom(\cK, \cF)^\vee$
to the cotangent complex $\L_Q$. 
\end{lemma}

\begin{proof}
In section \ref{_section_reduced_atiyah_}, we defined the reduced Atiyah class 
$$at\in\hom_{Q\times X}(\cK, \cF\otimes\pi^*\L_Q).$$ By
Grothendieck-Verdier duality and Serre duality, we have
\begin{eqnarray*}
\hom_{Q\times X}(\cK, \cF\otimes\pi^*\L_Q) &=& \hom_{Q\times X}(\crhom(\cF, \cK), \pi^*\L_Q)\\
                                                                        &=& \hom_Q(\rpils\crhom(\cF, \cK)[2], \L_Q)\\
                                                                        &=& \hom_Q(\rpils\crhom(\cK, \cF\otimes\omega_{\pi})^\vee, \L_Q)\\
                                                                        &=& \hom_Q(\rpils\crhom(\cK, \cF)^\vee, \L_Q).
\end{eqnarray*}
Therefore the class $at$ induces a morphism between the two complexes, denoted by
$$\gamma: \rpils\crhom(\cK, \cF)^\vee \lra \L_Q.$$
\end{proof}

The rest of the section is aiming at proving that, although $\gamma$ itself is not a quasi-isomorphism, if we replace 
the complex $\rpils\crhom(\cK, \cF)^\vee$ by its ``traceless" counterpart $\rpils\crhom(\cK, \cF)_0^\vee$ constructed in previous 
section, then we get a quasi-isomorphism. We start from the following lemma comparing the degree 0 cohomology groups. 

\begin{lemma}
The morphism $\gamma$ defined as above induces an isomorphism on the 0-th cohomology groups of the two complexes. 
\end{lemma}

This lemma was proved in \cite[Theorem 4.2]{Gillam-2011}. For the sake of completeness we include the proof here. 

\begin{proof}
For simplicity, in this proof we denote $$C:=\rpils\crhom(\cK, \cF)^\vee.$$
We are aiming to show that $$H^0(\gamma):H^0(C)\lra H^0(\L_Q)$$ is an isomorphism. 
By Yoneda's lemma for the abelian category of coherent sheaves, it suffices to show that, for every coherent sheaf $I$ on $Q$, 
the induces morphism $$H^0(\alpha)_I: \hom_Q(H^0(\L_Q), I)\lra\hom_Q(H^0(C), I)$$ is an isomorphism. 
However, notice that both complexes $\L_Q$ and $C$
have non-trivial cohomology only in non-positive degrees. Therefore, we have 
\begin{eqnarray*}
\hom_Q(H^0(\L_Q), I) &=& \hom_Q(\L_Q, I),\\
\hom_Q(H^0(C), I) &=& \hom_Q(C, I).
\end{eqnarray*}
Hence it suffices to show that the pullback morphism $$\gamma_I:\hom_Q(\L_Q, I)\lra\hom_Q(C, I)$$ is an isomorphism. 
Again by Grothendieck and Serre duality theorems, we have 
\begin{eqnarray*}
\hom_Q(C,I) &=& \hom_Q(\rpils\crhom(\cK, \cF)^\vee, I)\\
                 &=& \hom_Q(\rpils\crhom(\cF, \cK)[2], I)\\
                 &=& \hom_{Q\times X}(\crhom(\cF, \cK), \pi^*I)\\
                 &=& \hom_{Q\times X}(\cK, \cF\otimes\pi^*I).
\end{eqnarray*}
Therefore the morphism $\gamma_I$ becomes
$$\alpha_I: \hom_Q(\L_Q, I)\lra\hom_{Q\times X}(\cK, \cF\otimes\pi^*I),$$
which is given by the product with the reduced Atiyah class $at$. The deformation interpretation of this morphism
is given by Proposition \ref{_deformation_reduced_atiyah_},
namely, for any retraction $\iota: Q[I]\ra Q$, $\gamma_I(\iota)$ is a deformation of quotient of $\cE$, given by pulling back 
the universal quotient \ref{_universal_quotient_} from $Q$ to $Q[I]$
via $\iota$. 

However, because of the universal property of $Q$, any deformation of the universal quotient 
is obtained by pulling back from $Q$. Therefore, 
the above morphism $\alpha_I$ is an isomorphism, which concludes that $H^0(\alpha)$ is 
also an isomorphism. 
\end{proof}

We take the dual of the exact triangle \ref{_KF_traceless_} and get another exact triangle
$$\textrm{R}^2\pi_*\cO_{Q\times X}[1] \lra \rpils\crhom(\cK, \cF)^\vee \lra \rpils\crhom(\cK, \cF)^\vee_0.$$

\begin{lemma}
The morphism $$\gamma: \rpils\crhom(\cK, \cF)^\vee \lra \L_Q$$ can be lifted to a morphism
$$\gamma_0: \rpils\crhom(\cK, \cF)^\vee_0 \lra \L_Q.$$
\end{lemma}

\begin{proof}
It suffices to prove that $$\hom_Q(\textrm{R}^2\pi_*\cO_{Q\times X}[1], \L_Q)=0.$$
In fact, by \ref{_quot_lci_}, $\L_Q$ is quasi-isomorphic to a single sheaf in degree 0. Also notice that 
$\textrm{R}^2\pi_*\cO_{Q\times X}[1]=\cO_Q[1]$ is a single 
sheaf lying in degree $-1$, due to degree reason the above equation is true. Therefore $\gamma$ 
can be lifted to $\gamma_0$. 
\end{proof}

\begin{lemma}
The complex $\rpils\crhom(\cK, \cF)^\vee_0$ has perfect amplitude in $[-1, 0]$, and is quasi-isomorphic to a single sheaf in degree 0.
\end{lemma}

\begin{proof}
We first look at the fiber cohomology of $\rpils\crhom(\cK, \cF)_0$ before taking the dual. From the second 
part of Lemma \ref{_fiber_cohomology_}, we already know all cohomology groups when we restrict the 
complex $\rpils\crhom(\cK, \cF)_0$ to any closed point $p\in Q$. Moreover, by the long exact sequence induced 
by the restriction of the universal quotient at the point $p$, we can easily tell which of them vanish. It's not 
hard to find out that
\begin{equation*}
\ext^i(K_p,F_p)_0=
\begin{cases}
\hom(K_p,F_p) & \textrm{if } i=0;\\
\ext^2(F_p,F_p)_0 & \textrm{if } i=1;\\
0 & \textrm{otherwise}.
\end{cases}
\end{equation*}
Therefore, we know that the only possible non-trivial fiber cohomology lies in degree 0 and 1. Furthermore, over the locus where $F_p$ is 
stable, the only non-trivial fiber cohomology lies in degree 0. 

Now we turn to the dual complex $\rpils\crhom(\cK, \cF)^\vee_0$. Since the fiber cohomology respect the 
operation of taking duals, we conclude that, the only possible non-trivial fiber cohomology lies in degree $-1$ and 0. Furthermore, 
on the open subset of $Q$ where $F_p$ is stable, the fiber cohomology in degree $-1$ is even trivial. 

Now we are ready to prove the two statements in the lemma. First of all, we can always resolve the complex
$\rpils\crhom(\cK, \cF)^\vee_0$ by a perfect complex of finite length. We denote this perfect resolution by
$$A^s \lra A^{s+1} \lra \cdots \lra A^{t-1} \lra A^t.$$

We prove the first statement. If $s<-1$, we can actually truncate the complex at the position $s+1$, by replacing $A^s$ by 0
and $A^{s+1}$ by the cokernel of the map $A^s\lra A^{s+1}$. We claim that this cokernel is again a locally free
sheaf over $Q$. In fact, for any closed point $p\in Q$, the kernel of the fiber map $A^s_p\lra A^{s+1}_p$ is the fiber cohomology
group $\ext^{-s}(K_p, F_p)_0^\vee$, which by the above discussion vanishes when $s<-1$. Therefore the morphism 
between the two locally free sheaves $A^s$ and $A^{s+1}$ is fiberwise injective, hence has a locally free cokernel. 
This operation increases the lowest degree of the locally free resolution by 1. We can repeat this procedure until we have 
$s=-1$. 

Similarly, if $t>0$, we can always truncate the resolution step by step from the highest degree, while keeping every term 
in the complex locally free, until we reach $t=0$, by using the fact that the fiberwise cohomology groups vanish in positive
degrees. Therefore, we know tha the complex $\rpils\crhom(\cK, \cF)^\vee_0$ is quasi-isomorphic to a perfect complex
in degree $[-1, 0]$, which we still denote by $$A^{-1}\lra A^0.$$

Finally, to prove the second statement, we only need to show that the cohomology of this 2-term complex in degree $-1$ 
vanishes. In fact, we already know that over an open dense subset $Q^s$ of $Q$ where the quotient sheaf is stable, the fiberwise
cohomology of this 2-term cohomology is 0 in degree $-1$, which implies that the morphism of locally free sheaves 
$$A^{-1}\lra A^0$$ is injective over the stable locus $Q^s$. However any subsheaf of a locally free sheaf is torsion free, 
hence we conclude that the kernel sheaf is 0, which proves the second statement.
\end{proof}

Finally, we can state the main result of this section
\begin{theorem}\label{_quot_cotangent_complex_}
The morphism $$\gamma_0: \rpils\crhom(\cK, \cF)^\vee_0 \stackrel{\cong}{\lra} \L_Q$$ is a quasi-isomorphism. 
\end{theorem}

\begin{proof}
We have proved that both complexes have non-trivial cohomology groups only in degree 0, and $H^0(\alpha_0)$ is an isomorphism, 
from which the proposition is clear. 
\end{proof}

%=================================================
%       Section Five: Cotangent Complex of the Moduli Stack
%=================================================

\section{Cotangent Complex of the Moduli Stack}

The goal of this section is to compute the cotangent complex of the moduli stack $\cM=[Q/G]$. To achieve this goal, we will first build up
the commutative diagram
\begin{equation}\label{_traceless_commute_}
\xymatrix{
\rpils\crhom(\cK, \cF)_0^\vee \ar[r] \ar[d] & \rpils\crhom(\cE, \cF)_0^\vee \ar[d]\\
\L_Q \ar[r] & \L_{Q/\cM}
}
\end{equation}
such that the two vertical arrows are isomorphisms. We should notice that the two horizontal arrows are both functorial morphisms, 
while the left vertical arrow was constructed in the previous section and proved to be an isomorphism. It only remains to construct
the right vertical arrow to make the diagram commute, and prove that it's an isomorphism. 

Before we get into the main business, we need to study the fiber product of the quotient map $q:Q\lra\cM$ with itself. More precisely, we have
the following lemma: 

\begin{lemma}
The following diagram commutes: 
\begin{equation}\label{_fiber_product_}
\xymatrix{
G\times Q \ar@/^/[rrd]^m \ar@/_/[rdd]_{pr_2} \ar[rd]_j & & \\
 & Q\times_{\cM}Q \ar[r]_{p_1} \ar[d]^{p_2} & Q \ar[d]_q\\
 & Q \ar[r]^q & \cM,
}
\end{equation}
where $G$ is the gauge group $PGL(N)$, whose action on the Quot scheme $Q$ is the upper horizontal arrow $m$, $pr_i$ is the projection
from $G\times Q$ to the $i$-th factor, $p_i$ is the projection from $Q\times_{\cM} Q$ to the $i$-th factor, and $j=(m, pr_2)$. Moreover, 
$j$ is an isomorphism of schemes. 
\end{lemma}

\begin{proof}
The commutativity is straightforward. In fact, the commutativity of the square is due to the fiber product, and the commutativity of the 
two triangles is due to the definition of the map $j$. And the statement that $j$ is an isomorphism is a standard fact. For example, 
see \cite[Part I, Proposition 4.43]{Fantechi-2005}.
%also in \cite[Example 3.9]{Behrend-2006}.
\end{proof}

We will also need the following two facts related to the fiber product described in diagram \ref{_fiber_product_}.

\begin{lemma}\label{_cotangent_group_}
Notations are the same as above. Then we have $$m^*\L_{Q/\cM}=pr_1^*\L_G.$$
\end{lemma}

\begin{proof}
From the above lemma we know that the outer square of the diagram \ref{_fiber_product_} is also a fiber product. 
Since the quotient map $q$ is smooth, we apply \cite[Theorem 17.3 (4)]{Laumon-2000} and get
$$m^*\L_{Q/\cM}=\L_{G\times Q/Q}=pr_1^*\L_G,$$
which proves the claim.
\end{proof}

\begin{lemma}\label{_cotangent_factor_}
Notations are the same as above. Let $$\L_Q \lra \L_{Q/\cM}$$ be the canonical map induced by the right vertical arrow $q$,  
and $$m^*\L_Q \lra m^*\L_{Q/\cM}$$ be its pullback via the multiplication. Then we have the following commutative diagram
\begin{equation}
\xymatrix{
m^*\L_Q           \ar[r] \ar[d]        &              \L_{G\times Q}      \ar[d]      \\
m^*\L_{Q/\cM}  \ar[r]^{\cong}                &              pr_1^*\L_G
}
\end{equation}
where the upper horizontal arrow is the funtorial map induced by the multiplication map $m$, the lower horizontal arrow is the one 
constructed in Lemma \ref{_cotangent_group_}, and the right vertical arrow is the projection into the first factor. 
\end{lemma}

\begin{proof}
This is a direct application of the functoriality of the transitivity sequence in Lemma \ref{_cotangent_functorial_}.
\end{proof}

\begin{lemma}\label{_sheaf_identify_}
Notations are the same as above. Recall that $\cF$ is the universal quotient sheaf on $Q\times X$. Then we have 
$$p_1^*\cF=p_2^*\cF.$$ By further pulling back via the isomorphism $j$, we have $$m^*\cF=pr_2^*\cF.$$
\end{lemma}

\begin{proof}
This is a direct application of the construction of fiber products. 
Let's call $V=Q\times Q$ and denote the pullback of the universal family $\cF$ via the two projections from $V$ to $Q$ by 
$\cF_1$ and $\cF_2$. Then the fiber product $Q\times_{\cM} Q$ is defined to be 
the scheme $\textrm{Isom}(\cF_1, \cF_2)$, over which there's a canonical isomorphism from $p_1^*\cF$ to $p_2^*\cF$, 
which, by abuse of notation, was written as an equality in the lemma. 
\end{proof}

\begin{lemma}\label{_first_commutativity_}
The composition of the reduced Atiyah class and the functorial morphism of cotangent complexes factor through $\cE$. In other words, 
the dotted arrows in the following diagram exist and make the diagram commute: 
\begin{equation*}
\xymatrix{
\cK \ar@{-->}[r] \ar[d] & \cE \ar@{-->}[d]\\
\cF\otimes\pi^*\L_Q \ar[r] & \cF\otimes\pi^*\L_{Q/\cM}
}
\end{equation*}
\end{lemma}

\begin{proof}
To prove the composition of the two solid arrows factors through $\cE$, we only need to show that the composition of the following 
three maps is a zero map:
\begin{equation*}
\xymatrix{
\cF[-1] \ar[r] & \cK \ar[d] & \\
 & \cF\otimes\pi^*\L_Q \ar[r] & \cF\otimes\pi^*\L_{Q/\cM}
}
\end{equation*}
However, by Proposition \ref{_atiyah_compatibility_}, we know that the composition of the first two maps in the above diagram is exactly the
classical Atiyah class. Therefore the problem becomes to show the composition of the classical Atiyah class and the functorial morphism
between the cotangent complexes, i.e., the following two morphism, is a zero map: 
\begin{equation}\label{_atiyah_cotangent_}
\cF[-1] \lra \cF\otimes\pi^*\L_Q \lra \cF\otimes\pi^*\L_{Q/\cM}.
\end{equation}

We can pull back the maps in equation \ref{_atiyah_cotangent_} via the multiplication $m$. If we denote $m^*\cF$ by $\tilde{\cF}$, 
by applying Lemma \ref{_cotangent_group_}, we get
\begin{equation}\label{_atiyah_cotangent_pullback_}
\tilde{\cF}[-1] \lra \tilde{\cF}\otimes m^*\L_Q \lra \tilde{\cF}\otimes pr_1^*\L_G.
\end{equation}
We will first show that the compositions of these two maps is zero. 

By Lemma \ref{_cotangent_factor_}, we can further replace the above maps into the composition of three
\begin{equation}\label{_three_composition_}
\tilde{\cF}[-1] \lra \tilde{\cF}\otimes m^*\L_Q \lra \tilde{\cF}\otimes \L_{G\times Q} \lra \tilde{\cF}\otimes pr_1^*\L_G.
\end{equation}

Since $\tilde{\cF}$ is obtained by the pullback via $m$, by the functorial property of Atiyah classes, we realized that 
the composition of the first two maps in \ref{_three_composition_} is simply the Atiyah
class of the sheaf $\tilde{\cF}$ itself! 

However, by Lemma \ref{_sheaf_identify_}, we see that the universal sheaf $\tilde{\cF}$ can also be realized as $pr_2^*\cF$, therefore 
by the functorial property again, its Atiyah class can also be viewed as the pull back of the Atiyah class of $\cF$ via the projection $pr_2$. 
In particular, it lies in the second component of 
$$\ext^1(\tilde{\cF}, \tilde{\cF}\otimes\L_{G\times Q}) = \ext^1(\tilde{\cF}, \tilde{\cF}\otimes pr_1^*\L_G) 
\oplus \ext^1(\tilde{\cF}, \tilde{\cF}\otimes pr_2^*\L_Q).$$
Therefore its projection into the first factor is 0, which implies the composition of the three maps in \ref{_three_composition_} is 0. 

Now we define the diagonal map 
\begin{eqnarray*}
\Delta: Q &\lra& G\times Q,\\
q &\longmapsto& (1,q)
\end{eqnarray*}
and it's easy to see that the composition $m \circ \Delta$ is the identity
map on $Q$. Because of this, we can pull back the maps in \ref{_atiyah_cotangent_pullback_} via the map $\Delta$ and 
obtain the original maps in \ref{_atiyah_cotangent_}. The above discussion implies that the composition in 
\ref{_atiyah_cotangent_} is 0. Therefore the dotted arrows in the statement exist and make the whole diagram commutative, where 
the horizontal dotted arrow is simply the one in the exact sequence of the universal quotient over the Quot scheme $Q$. 
\end{proof}

We can translate the above lemma into the language of cotangent complexes as follows. 

\begin{lemma}
We have the commutative diagram
\begin{equation}\label{_trace_commute_}
\xymatrix{
\rpils\crhom(\cK,\cF)^\vee    \ar[r]\ar[d]          &              \rpils\crhom(\cE,\cF)^\vee      \ar[d]           \\
\L_Q                                     \ar[r]                   &              \L_{Q/\cM},
}
\end{equation}
where the upper horizontal arrow is the natural map from the universal quotient, the lower horizontal arrow is the functorial map
given by the quotient, and the left vertical arrow is given by the reduced Atiyah class. 
\end{lemma}

\begin{proof}
This is just a literal translation of lemma \ref{_first_commutativity_}. In the proof of Lemma \ref{_reduced_atiyah_map_}, we use the 
Grothendieck-Verdier duality and Serre duality to construct a canonical isomorphism (by abuse of notation we simply use
equalities)
$$\hom_{Q\times X}(\cK, \cF\otimes\pi^*\L_Q)=\hom_Q(\rpils\crhom(\cK, \cF)^\vee, \L_Q).$$
Following exactly the same steps, we can construct another canonical isomorphism
$$\hom_{Q\times X}(\cE, \cF\otimes\pi^*\L_{Q/\cM})=\hom_Q(\rpils\crhom(\cE, \cF)^\vee, \L_{Q/\cM}).$$
If we denote the two vertical arrows in Lemma \ref{_first_commutativity_} by $u$ and $v$, which are elements of the spaces
on the left hand side of the two equations respectively. We denote the corresponding elements on the right hand side 
by $u'$ and $v'$. 

The previous lemma claims that, the composition of $u$ with the canonical map $\L_Q\lra \L_{Q/\cM}$
agrees with the precomposition of $v$ with the map $\cK\lra\cE$ in the universal quotient sequence \ref{_universal_quotient_}. 
Therefore, by the above canonical isomorphisms, we know that, the composition of $u'$ with the canonical map $\L_Q\lra \L_{Q/\cM}$
also agrees with the precomposition of $v'$ with the map $\cK\lra\cE$ in the universal quotient sequence \ref{_universal_quotient_},
which is exactly the conclusion of this lemma. 
\end{proof}

If we compare the above result with the one we stated at the beginning of the section, we need to replace the upper two complexes
by their traceless counterparts. Therefore we have the following lemma. 

\begin{lemma}
We have the commutative diagram \ref{_traceless_commute_}.
\end{lemma}

\begin{proof}
From the discussion in section 3 we had the following two exact triangles, 
which are the dual of the exact triangles \ref{_EF_traceless_} and \ref{_KF_traceless_}:
\begin{equation}\label{_KF_dual_traceless_}
\cO_Q[1] \lra \rpils\crhom(\cK,\cF)^\vee \lra \rpils\crhom(\cK,\cF)_0^\vee
\end{equation}
and
\begin{equation}\label{_EF_dual_traceless_}
\rpils\crhom(\cE,\cF)_0^\vee \lra \rpils\crhom(\cE,\cF)^\vee \lra \cO_Q.
\end{equation}

First of all we claim that both arrows coming out of $\rpils\crhom(\cK,\cF)^\vee$ in \ref{_trace_commute_} factor through 
$\rpils\crhom(\cK,\cF)_0^\vee$. For this purpose it suffices to show that the pre-composition of these two arrows by the 
first arrow in \ref{_KF_dual_traceless_} is 0. In fact, since $\cO_Q[1]$ is a single sheaf lying in degree $-1$, while both 
$\L_Q$ and $\rpils\crhom(\cK,\cF)^\vee$ 
are both single sheaves lying in degree 0, there is only the zero map from a sheaf in degree $-1$ to a sheaf in degree 0. 
This allows us to replace the upper left corner of \ref{_trace_commute_} by its traceless counterpart. 

Next we claim that the upper horizontal arrow in \ref{_trace_commute_} can be lifted to $\rpils\crhom(\cE,\cF)_0^\vee$. 
For this we only need to show, that the composition of this arrow with the second arrow in \ref{_EF_dual_traceless_}
$$\rpils\crhom(\cK,\cF)^\vee \lra \rpils\crhom(\cE,\cF)^\vee \lra \cO_Q$$ is a zero map, or equivalently, its dual composition
$$\cO_Q \lra \rpils\crhom(\cE,\cF) \lra \rpils\crhom(\cK,\cF)$$ is a zero map on $Q$. 
Since we assume that there is at least one stable sheaf in the moduli space, the stable locus $Q^s$ in the Quot scheme is open
and dense. Therefore it suffices to check the above claim at every closed point in $Q^s$. 

Pick any closed point $p\in Q^s$. By the construction of the traceless complexes, the restriction of the above two maps at $p$
becomes
$$\hom(F_p,F_p) \lra \hom(E_p,F_p) \lra \hom(K_p,F_p),$$ whose composition of 0, as expected. Therefore we can as well
replace the upper right corner of \ref{_trace_commute_} by its traceless counterpart and obtain the commutative diagram 
\ref{_traceless_commute_}. 
\end{proof}

Finally we are aiming to prove that the right vertical map in \ref{_traceless_commute_} is a quasi-isomorphism, or more precisely, 
an isomorphism between two single sheaves in degree 0. First we compute the two sheaves explicitly to see if they have a chance 
to be isomorphic. 

\begin{lemma}
Both $\rpils\crhom(\cE,\cF)_0^\vee$ and $\L_{Q/\cM}$ are quasi-isomorphic to a trivial vector bundle of rank equal to $\dim G$
concentrated in degree 0.
\end{lemma}

\begin{proof}
We recall the construction of the Quot scheme. There exists a sufficient large integer $n$, such that for every semistable sheaf $F$ with 
the prescribed Mukai vector, $F\otimes \cO(n)$ has trivial cohomology in positive degrees, and $E\otimes cO(n)$ is the trivial bundle 
generated by the global sections of $F\otimes \cO(n)$. Assuming the dimension of the global sections is $N$, then the gauge group 
$G=PGL(N)$. Therefore we have 
\begin{eqnarray*}
\rpils\crhom(\cE,\cF) &=& \rpils\crhom(\cE(n),\cF(n)) \\
                                 &=& \rpils\crhom(\cO^{\oplus N},\cF(n)) \\
                                 &=& \rpils\cF(n)\otimes \cO^{\oplus N \vee} \\
                                 &=& \pi_*\cF(n)\otimes \cO^{\oplus N \vee} \\
                                 &=& \cO^{\oplus N \vee} \otimes \cO^{\oplus N} = \cend(\cO^{\oplus N}).
\end{eqnarray*}
And from the construction of exact triangle \ref{_EF_traceless_}, we see that the map $\cO_Q \lra \rpils\crhom(\cE,\cF)$ is at 
every closed point $p\in Q$ given by the identity map $$\C\textrm{Id} \hookrightarrow \hom(F_p,F_p) \lra \hom(E_p,F_p)$$  which
is injective. Therefore the exactly triangle \ref{_EF_traceless_} actually becomes an exact sequence of sheaves on $Q$
$$0 \lra \cO_Q \lra \cend(\cO_Q^{\oplus N}) \lra \cend(\cO_Q^{\oplus N})_0 \lra 0.$$ So $\rpils\crhom(\cE,\cF)_0$ is quasi-isomorphic to
$\cend(\cO_Q^{\oplus N})_0$ which is a trivial bundle of rank $N^2-1$ concentrated in degree 0. 

On the other hand, by noticing that $$m \circ \Delta = \textrm{Id},$$
together with Lemma \ref{_cotangent_group_}, we have 
$$\L_{Q/\cM} = \Delta^* m^* \L_{Q/\cM} = \Delta^* pr_1^* \L_G = \mathfrak{g} \otimes \cO_Q$$
which is also a trivial bundle of rank equal to $\dim\mathfrak{g}=N^2-1$.
\end{proof}

Finally, we are ready to prove the following lemma. 

\begin{lemma}
The right vertical arrow constructed in \ref{_traceless_commute_}
$$\varphi: \rpils\crhom(\cE,\cF)^\vee_0 \lra \L_{Q/\cM}$$
is an isomorphism of two sheaves. 
\end{lemma}

\begin{proof}
From the above discussion we know that this arrow is a map between two locally free sheaves of the same rank. 

First of all we will show that, on the stable locus $Q^s$, $\varphi$ is an isomorphism. For this purpose, 
it suffices to show that $\varphi_p$ is surjective on the stable locus $Q^s$. 
However, due to the commutativity of the diagram \ref{_traceless_commute_}, whose left vertical arrow 
is an isomorphism, it suffices to show that the functorial map 
\begin{equation}\label{_cotangent_map_stable_}
\L_{Q^s} \lra \L_{Q^s/\cM^s}
\end{equation}
is surjective 
on $Q^s$, where $\cM^s$ is as a substack of $\cM$ the quotient of $Q^s$ by the group $G$. 

To show that the map \ref{_cotangent_map_stable_} is surjective, we only need to show that the pullback of the 
map via 
$$m_s: G\times Q^s \lra Q^s$$ is surjective. By applying Lemma \ref{_cotangent_group_}, we just need to
prove that $$m_s^*\L_{Q^s} \lra pr_1^*\L_G$$ is surjective. Here by abuse of notation, we use $pr_1$ for
the projection of $G\times Q^s$ to its first factor. 

Since both $Q^s$ and $G$ are smooth, we can consider the dual of the above map $$pr_1^*T_G \lra m^*T_{Q^s}.$$
We need to show that it's injective on fibers at every closed point $p\in Q^s$. Or in other words, we need to show that
the pushforward of the tangent spaces $$m_{s*}(pr_1^*T_G) \lra T_{Q^s}$$ is injective at every closed point $p\in Q^s$. 
However, this is equivalent of saying that the $G$-action is free on the stable locus $Q^s$, which is obvious. 

So far we have proved that the map $\varphi$ is an isomorphism of two locally free sheaves of the same rank on $Q^s$. 
Next we claim that $\varphi$ is actually an isomorphism over $Q$. In fact, the locus in $Q$ where $\varphi$ is not an isomorphism
is the zero locus of the corresponding map of determinant line bundles, therefore is a Cartier divisor. In particular, if it's not
an empty set, it should have dimension 1. However, by \cite[Proposition 6.1]{Kaledin-2006} that the strictly semistable locus 
$Q\backslash Q^s$ has codimension at least 2. Therefore the degeneracy locus must be empty, and $\varphi$ is 
an isomorphism everywhere. 
\end{proof}

Now we get out key result on the cotangent complex of the moduli stack. 

\begin{theorem}\label{_stack_cotangent_complex_}
We have a quasi-isomorphism $$\rpils\crhom(\cF,\cF)_0^\vee[-1] \stackrel{\cong}{\lra} q^*\L_\cM.$$
\end{theorem}

\begin{proof}
From the above discussion on the diagram \ref{_traceless_commute_}, and two functorial exact triangles, we obtain the following
diagram (in which the first exact triangle follows from equation \ref{_traceless_triangle_}):
\begin{equation}
\xymatrix{
\rpils\crhom(\cK,\cF)_0^\vee  \ar[r]\ar[d]     &      \rpils\crhom(\cE,\cF)_0^\vee    \ar[r]\ar[d]      &       \rpils\crhom(\cF,\cF)_0^\vee    \ar@{-->}[d]   \\
\L_Q             \ar[r]                                  &              \L_{Q/\cM}               \ar[r]                        &                q^*\L_{\cM}[1]
}
\end{equation}
Since the left square commutes, by the axioms of triangulated categories, the dotted arrow exists and is a quasi-isomorphism. 
\end{proof}

%=================================================
%           Section Six: Symplectic Complexes and Symplectic Stacks
%=================================================

%=================================================
%           Plan: 1. definition of symplectic complexes
%                   2. definition of symplectic stacks
%                   3. trivial example: finite group quotient of a symplectic manifold
%                   4. main example: moduli of sheaves on K3
%=================================================

\section{Symplectic Stacks}

Motivated by the symmetric obstruction theory in \cite{Behrend-2008}, we want to study bilinear pairings on complexes. 
The following notion of anti-symmetric forms is completely parallel to \cite[Definition 1.1]{Behrend-2008}:

\begin{definition}\label{_definition_complex_}
Let $\cX$ be a scheme, and $E^\hdot\in D^b(\cX)$ be a perfect complex. A \textit{non-degenerate anti-symmetric bilinear form}
on $E^\hdot$ is a morphism $$\beta: E^\hdot\otimes E^\hdot \lra \cO_{\cX} $$ in $D^b(\cX)$, which is
\begin{enumerate}
\item anti-symmetric, i.e. the following diagram is commutative
\begin{equation}\label{_tensor_anticommutativity_}
\xymatrix{
E^\hdot \otimes E^\hdot \ar^{\beta}[r] \ar^{\iota}[d] & \cO_{\cX} \ar^{-\textrm{id}}[d] \\
E^\hdot \otimes E^\hdot \ar^{\beta}[r] & \cO_{\cX},
}
\end{equation}
where $\iota$ is the isomorphism switching the two factors of the tensor product;
\item non-degenerate, which means that $\beta$ induces an isomorphism $$\theta: E^\hdot \lra E^{\hdot \vee}.$$
\end{enumerate}
In such a case, we call $E^\hdot$ a \textit{symplectic complex} and $\beta$ a \textit{symplectic pairing} on $E^\hdot$. 
\end{definition}

\begin{remark}
Note that there are other equivalent ways of phrasing this definition (c.f. \cite[Remark 1.2]{Behrend-2008}). In fact, we can 
avoid using the tensor product and use only the isomorphism $\beta$, then the condition of anti-symmetry becomes 
$\theta^\vee=-\theta$, or more precisely, the following diagram commutes:
\begin{equation}\label{_dual_new_check_}
\xymatrix{
E^\hdot \ar^{\theta}[r] \ar^{i_E}[d] & E^{\hdot \vee} \ar^{-\textrm{id}}[d] \\
E^{\hdot \vee \vee} \ar^{\theta^\vee}[r] & E^{\hdot \vee},
}
\end{equation}
where $i$ is the naturally isomorphism of the perfect complex $E$ and its double dual. 

Similar to the situation in \cite{Behrend-2008}, it's usually easier to work with $\theta$ only. Then 
an anti-symmetric pairing on the complex $E$ is simply an isomorphism $\theta: E^\hdot\lra E^{\hdot \vee}$ satisfying $\theta^\vee=-\theta$. 
\end{remark}

\begin{remark}
Note that here we adopted the sign conventions in \cite[Section 1.3]{Conrad-2000}. The sign conventions which are most 
relevant to the above definition are the ones related to switching the two factors in a tensor product and to the identification 
of a complex with its dual. More precisely, we should keep in mind that the definition of the natural isomorphism
\begin{equation}\label{_tensor_product_commute_}
E^\hdot_1\otimes E^\hdot_2\cong E^\hdot_2\otimes E^\hdot_1
\end{equation}
uses a sign of $(-1)^{pq}$ on the component $E_1^p\otimes E_2^q$ 
\cite[page 11]{Conrad-2000}. Moreover, from the definition of $\chom$ complex in \cite[page 10]{Conrad-2000}, we can 
easily find that if $E$ is a prefect complex represented by 
$$\cdots \lra E^{i-1} \stackrel{\varphi_{i-1}}{\lra} E^i \stackrel{\varphi_i}{\lra} E^{i+1} \lra \cdots,$$ 
then the dual complex $E^\vee$ can be represented by 
$$\cdots \lra (E^{i+1})^\vee \stackrel{(-1)^i\varphi_i^\vee}{\lra} (E^i)^\vee \stackrel{(-1)^{i-1}\varphi_{i-1}^\vee}{\lra} (E^{i-1})^\vee \lra \cdots,$$ 
and the double dual complex $E^{\vee\vee}$ becomes
$$\cdots \lra E^{i-1} \stackrel{-\varphi_{i-1}}{\lra} E^i \stackrel{-\varphi_i}{\lra} E^{i+1} \lra \cdots.$$ 
Note that the extra sign is induced in all the morphisms in the complex. To get compatible with this, according to \cite[page 14]{Conrad-2000}, 
the isomorphism $i_E: E^\hdot\lra E^{\hdot\vee\vee}$ is chosen to involve a sign of $(-1)^n$ in degree $n$. 
\end{remark}

An obvious example of a symplectic complex is a single vector bundle equipped with a symplectic metric sitting in degree 0. However, 
to get a better feeling of a symplectic complex, especially the tricky sign conventions, we can see the following example: 

\begin{example}
Let $X=\C^{2n}$ with $x_1, x_2, \cdots, x_n, y_1, y_2, \cdots, y_n$ as coordinates. Let $E$ be the complex of locally free sheaves
$$\cO_X \stackrel{\alpha}{\lra} \cO^{\oplus 2n}_X \stackrel{\beta}{\lra} \cO_X,$$
where the morphisms are
\begin{eqnarray*}
\alpha &=& (x_1, \cdots, x_n, -y_1, \cdots, -y_n)^T,\\
\beta &=& (y_1, \cdots, y_n, x_1, \cdots, x_n), 
\end{eqnarray*}
where the letter ``T" in the upper right corner denotes the transpose of the matrix. 
Then the dual complex $E^\vee$ becomes
$$\cO_X \stackrel{\beta^T}{\lra} \cO^{\oplus 2n}_X \stackrel{-\alpha^T}{\lra} \cO_X,$$
and we can define a morphism $\theta: E\lra E^\vee$ by 
\begin{equation*}
\xymatrix{
\cO_X \ar^{\alpha}[r] \ar^{\textrm{id}}[d] & \cO_X^{\oplus 2n} \ar^{\beta}[r] \ar^{\gamma}[d] & \cO_X \ar^{\textrm{id}}[d]\\
\cO_X \ar^{\beta^T}[r] & \cO_X^{\oplus 2n} \ar^{-\alpha^T}[r] & \cO_X,
}
\end{equation*}
where $\gamma$ is the standard $2n\times 2n$ symplectic matrix
\begin{equation*}
\left(
\begin{array}{cc}
0 & -1_n\\
1_n & 0
\end{array}
\right).
\end{equation*}
The dual isomorphism $\theta^\vee: E^{\vee\vee}\lra E^{\vee}$ now becomes
\begin{equation*}
\xymatrix{
\cO_X \ar^{-\alpha}[r] \ar^{\textrm{id}}[d] & \cO_X^{\oplus 2n} \ar^{-\beta}[r] \ar^{\gamma^T}[d] & \cO_X \ar^{\textrm{id}}[d]\\
\cO_X \ar^{\beta^T}[r] & \cO_X^{\oplus 2n} \ar^{-\alpha^T}[r] & \cO_X,
}
\end{equation*}
We also mentioned above that the natural isomorphism $i_E:E\lra E^{\vee\vee}$ is defined to be 
\begin{equation*}
\xymatrix{
\cO_X \ar^{\alpha}[r] \ar^{-\textrm{id}}[d] & \cO_X^{\oplus 2n} \ar^{\beta}[r] \ar^{\textrm{id}}[d] & \cO_X \ar^{-\textrm{id}}[d]\\
\cO_X \ar^{-\alpha}[r] & \cO_X^{\oplus 2n} \ar^{-\beta}[r] & \cO_X.
}
\end{equation*}
The above diagrams verify the required symplectic condition in \ref{_dual_new_check_}. Therefore the complex $E$ in this example is
a symplectic complex. 
\qed
\end{example}

Now we define a symplectic complex on an algebraic stack, by using an atlas of a stack. 

\begin{definition}\label{_bad_definition_for_stack_}
Let $\cX$ is an algebraic stack, and let $u:U\lra \cX$ be an atlas of the stack $\cX$, where $U$ is a scheme. Let 
$\mathcal{G}\in D^b(\cX)$ be a perfect complex. We say $\mathcal{G}$ is a \emph{symplectic complex}, if there exists 
a symplectic pairing $$\beta: u^*\mathcal{G}\otimes u^*\mathcal{G} \lra \cO_U, $$ satisfying that 
$$q_1^*\beta=q_2^*\beta,$$ where $q_1$ and $q_2$ are the projections in the following fiber diagram
\begin{equation*}
\xymatrix{
U\times_{\cX}U \ar^{q_1}[r] \ar_{q_2}[d] & U \ar^u[d] \\
U \ar_u[r] & \cX.
}
\end{equation*}
\end{definition}

Based on the definition of symplectic complex, we can now define the following notion of symplectic stacks: 

\begin{definition}\label{_definition_stack_}
Let $\cX$ be a scheme or an algebraic stack. We call $\cX$ a \textit{symplectic stack}, if its cotangent complex $\L_{\cX}$ is a symplectic complex. 
\end{definition}

From this definition we immediately see: 

\begin{example}
Any holomorphic symplectic manifold $X$ is a symplectic stack, because a nowhere degenerate
holomorphic 2-form defines a symplectic pairing on the tangent bundle $T_X$, or equivalently the cotangent bundle $\Omega_X$. 
\end{example}

A slightly more general situation is the following: 

\begin{example}
Let $X$ be a holomorphic symplectic manifold with a nowhere degenerate holomorphic 2-form $\sigma$, and $G$ is a finite group 
acting on $X$ preserving the symplectic form $\sigma$. Let $q: X\lra\cX=[X/G]$ be the stacky quotient map. Then the Deligne-Mumford stack 
$\cX$ is a symplectic stack. 
\end{example}

In fact, by Proposition \ref{_cotangent_functorial_}, we know that $$q^*\L_{\cX}=\L_X=\Omega_X,$$ because $G$ is finite. The holomorphic
symplectic form $\sigma$ defines a symplectic pairing on $\Omega_X$. Since the $G$-action preserves $\sigma$, this symplectic
pairing descends to $\L_{\cX}$, which shows $\cX$ is a symplectic stack. 

We mentioned that the cotangent complex of a stack could lie over all degrees not larger than 1. However, for a symplectic stack, 
due to the isomorphism between the cotangent complex and its dual, its perfect amplitude could only be within the interval $[-1,1]$. 
Therefore, the cotangent complex could have only two types: either a single locally free sheaf sitting in degree 0, 
or a perfect complex in degree $[-1,1]$. 
The above examples fall in the first type. However, all the calculations from previous sections provide us examples of the second type. 

\begin{theorem}\label{_moduli_symplectic_stack_}
The moduli stack $\cM$ of semistable sheaves on a K3 surface is a symplectic stack.
\end{theorem}

\begin{proof}
By Proposition \ref{_stack_cotangent_complex_}, we know that the pullback of the cotangent complex via the quotient map is
$$\rpils\crhom(\cF,\cF)_0^\vee[-1] \stackrel{\cong}{\lra} q^*\L_{\cM}.$$
To prove the cotangent complex $\L_{\cM}$ is a symplectic complex, we first show that there exists a symplectic 
pairing on $q^*\L_{\cM}$, then show that the symplectic pairing satisfies the compatibility condition in Definition 
\ref{_bad_definition_for_stack_} for a symplectic complex on a stack. 

The relative Serre duality tells us that the composition of the derived Yoneda product and the trace map
\begin{eqnarray*}
&  & \rpils\crhom(\cF,\cF) \otimes \rpils\crhom(\cF,\cF\otimes\omega_{\pi}) \\
& \stackrel{\cup}{\lra} & \textrm{R}^2\pi_*\crhom(\cF,\cF\otimes\omega_{\pi})[-2] \\
& \stackrel{\textrm{tr}}{\lra} & \textrm{R}^2\pi_*\omega_{\pi}[-2] \cong \cO_Q[-2]
\end{eqnarray*}
is a non-degenerate bilinear form. 

Due to the fact that $X$ is a K3 surface, the relative dualizing sheaf $\omega_{\pi}$ of the projection $\pi: Q\times X\lra Q$ 
has a trivialization given by the pullback of generator of $H^{2,0}(X)$ via the second projection. We denote the isomorphism by
$$\sigma: \cO_{Q\times X} \lra \omega_{\pi}.$$
Then we can also write the above non-degenerate bilinear form as
\begin{eqnarray*}
&  & \rpils\crhom(\cF,\cF) \otimes \rpils\crhom(\cF,\cF) \\
& \stackrel{\cup}{\lra} & \textrm{R}^2\pi_*\crhom(\cF,\cF)[-2] \\
& \stackrel{\textrm{tr}}{\lra} & \cO_Q[-2]
\end{eqnarray*}
Furthermore, this trace map also satisfies the symmetry condition (\cite[Equation 10.3]{Huybrechts-2010b})
$$\textrm{tr}(e\cup e')=(-1)^{\deg (e) \deg (e')}\textrm{tr}(e'\cup e). $$
After a degree shift we get a bilinear form
$$ \rpils\crhom(\cF,\cF)[1] \otimes \rpils\crhom(\cF,\cF)[1] \lra \cO_Q $$
which is still non-degenerate, and the symmetry condition becomes
\begin{eqnarray*}
\textrm{tr}(e\cup e') &=& (-1)^{(\deg (e)+1) (\deg (e')+1)}\ \textrm{tr}(e'\cup e) \\
                               &=& (-1)^{\deg (e) \deg (e') + \deg (e) + \deg (e') +1}\ \textrm{tr}(e'\cup e)\\
                               &=& (-1)^{\deg (e) \deg (e') +1}\ \textrm{tr}(e'\cup e).
\end{eqnarray*}
The reason for the last equality is that: for $\textrm{tr}(e \cup e')$ to lie in the only non-trivial degree of the 
complex $\cO_Q$, we must have $$\deg (e) + \deg (e')=0.$$ 
Comparing the above equation with the sign convention in the equation \ref{_tensor_product_commute_}, 
we realize that, switching the two factors in the trace map actually introduces an extra negative sign.
This verifies the condition in equation \ref{_tensor_anticommutativity_},
therefore the bilinear pairing on $\rpils\crhom(\cF,\cF)[1]$ is anti-symmetric. 

It's also clear that the restriction of the above symplectic pairing on the traceless complex 
$\rpils\crhom(\cF,\cF)_0[-1]$ again defines a symplectic pairing. It suffices to show that it's still non-degenerate.
In fact, the above application of Serre duality can also be written in the form of
\begin{equation}\label{_second_interpretation_}
\rpils(\crhom(\cF,\cF))\otimes \rpils\crhom(\crhom(\cF, \cF), \omega_{\pi}) \lra \textrm{R}^2\pi_*\omega_{\pi}[-2],
\end{equation}
which can also be thought as the relative 
Serre duality on a single sheaf $\crhom(\cF,\cF)$. 
Note that from the decomposition 
\begin{equation}\label{_trace_decomposition_old_}
\crhom(\cF,\cF)=\crhom(\cF,\cF)_0\oplus \cO_{Q\times X}
\end{equation}
we also get
$$\crhom(\cF,\cF)^\vee=\crhom(\cF,\cF)_0^\vee\oplus \cO_{Q\times X}^\vee.$$ Together with the isomorphism 
$\omega_{\pi}=\cO_{Q\times X}$, we immediately obtain that $\rpils\crhom(\crhom(\cF, \cF)_0, \omega_{\pi})$ is also 
naturally a direct summand of $\rpils\crhom(\crhom(\cF, \cF), \omega_{\pi})$. Therefore, the restriction of 
equation \ref{_second_interpretation_} on the traceless complex becomes
\begin{equation}
\rpils(\crhom(\cF,\cF)_0)\otimes \rpils\crhom(\crhom(\cF, \cF)_0, \omega_{\pi}) \lra \textrm{R}^2\pi_*\omega_{\pi}[-2],
\end{equation}
which can be thought as the Serre duality on a single sheaf $\crhom(\cF,\cF)_0$, therefore is again non-degenerate. Hence
after shifting, we get a symplectic pairing on $\rpils\crhom(\cF,\cF)_0[1]$, and by taking dual we get a symplectic pairing on
$\crhom(\cF,\cF)_0^\vee[-1]$.

So far we have proved that $q^*\L_{\cM}$ is a symplectic complex. 

Finally, we want to show that the symplectic pairing on the complex $\rpils\crhom(\cF,\cF)_0$ is $G$-equivariant. For this 
purpose we just need to show that, the pull back of the symplectic pairing via the maps $m$ and $pr_2$ in diagram 
\ref{_fiber_product_} agree with each other. 

By flatness and 
\cite[Proposition 5.8, 5.9]{Hartshorne-1966}, as well as the fact that the decomposition \ref{_trace_decomposition_old_}
is natural under pullback, we conclude that the pullback of the Serre duality pairing
\begin{equation}\label{_before_pull_back_}
\rpils\crhom(\cF,\cF)_0\otimes\rpils\crhom(\cF,\cF)_0\lra \textrm{R}^2\pi_*\cO_Q[-2]
\end{equation}
via $m$ is 
$$\rpils\crhom(m^*\cF,m^*\cF)_0\otimes\rpils\crhom(m^*\cF,m^*\cF)_0\lra \textrm{R}^2\pi_*\cO_{G\times Q}[-2],$$
which is again the Serre duality pairing on $G\times Q$ by the functoriality of Serre duality. 

Similarly, if we pull back the pairing \ref{_before_pull_back_} via the other map $pr_2$, we again get a Serre duality pairing
$$\rpils\crhom(pr_2^*\cF,pr_2^*\cF)_0\otimes\rpils\crhom(pr_2^*\cF,pr_2^*\cF)_0\lra \textrm{R}^2\pi_*\cO_{G\times Q}[-2].$$

In Lemma \ref{_sheaf_identify_}, we have showed that the pullback of the universal sheaf $\cF$ via $m$ and $pr_2$ are 
canonically isomorphic, denoted by $$\tilde{\cF}=m^*\cF=pr_2^*\cF.$$ Therefore the above two pullback maps agree with 
each other, and we conclude that the moduli stack $\cM$ of the semistable sheaves on a K3 surface is a symplectic stack in
the sense of Definition \ref{_definition_stack_}.

\end{proof}

\bibliographystyle{alpha}
\bibliography{references}

\end{document}